\documentclass[a4paper,11pt]{amsart}
\usepackage{graphicx} 
\usepackage{amsmath,amssymb,amsthm} 
\usepackage{tikz-cd}
\usepackage{mathtools}
\usepackage{enumitem}
\usepackage{hyperref}
\newtheorem{prop}{Proposition}
\newtheorem{theorem}{Theorem}
\newtheorem{corollary}{Corollary}
\newtheorem{lemma}{Lemma}
\newtheorem{remark}{Remark}

\DeclareMathOperator{\HF}{\text{HF}}

\theoremstyle{plain}
\newcommand{\thistheoremname}{}
\newtheorem*{genericthm*}{\thistheoremname}
\newenvironment{namedthm*}[1]
  {\renewcommand{\thistheoremname}{#1}%
   \begin{genericthm*}}
  {\end{genericthm*}}

\title[Independence of generic forms and the Fr\"oberg conjecture]{Independence of generic forms and the Fr\"oberg conjecture}
\author{Mats Boij}
\address{Department of Mathematics, KTH - Royal Institute of Technology,  SE-100 44 Stockholm, Sweden}
\email{boij@kth.se}
\author{Eric Dannetun}
\address{Department of Mathematics, Stockholm University, SE-106 91 Stockholm, 
Sweden}
\email{eric.dannetun@math.su.se}
\author{Samuel Lundqvist}
\address{Department of Mathematics, Stockholm University, SE-106 91 Stockholm, 
Sweden}
\email{samuel@math.su.se}

\date{}

\newcommand{\compBound}{2100}

\begin{document}

\begin{abstract}
We show that the Fr\"oberg conjecture holds in the second non-trivial degree for an ideal generated by generic forms of degree $d>2$. We also show that the conjecture is true up to degree $2d-1$ provided that the number of variables is sufficiently large.
\end{abstract}

\maketitle

\section{Introduction}

Let $k$ be an infinite field and let $R=k[x_1,\ldots,x_n]$. Let $I$ be an ideal generated by $r$ generic forms of degrees $d_1,\ldots, d_r$. Fr\"oberg 
has conjectured that the Hilbert series of $R/I$ equals
\begin{equation} \label{eq:froberg}
\left[ \frac{\prod_{i=1}^r (1-t^{d_i})}{(1-t)^n} 
\right],
\end{equation}
where the brackets notation means truncate at the first non-positive coefficient. 

The conjecture 
is trivially true for $r\leq n$, settled for $n=2$ by Fr\"oberg \cite{Froberg1985}, for $n=3$ by Anick \cite{ANICK1986235}, and for $r=n+1$ independently by Stanley  \cite{Stanley1980WeylGT} and Watanabe \cite{WATANABE1989194}. 

In addition to the results for a specific number of variables or a specific number of forms, there are also results on specific coefficients in the series for the equigenerated case $d= d_1= \cdots = d_r$. By the important partial result of Hochster and Laksov \cite{Hochster01011987}, there are no unexpected linear syzygies of ideals generated by forms of degree $d\geq 2$, meaning that the Fr\"oberg conjecture is true in degree $d+1.$ 
Moreover, Aubry \cite{Aubry1995392}, Migliore and Mir\'o-Roig \cite{MIGLIORE200379}, and Nenashev \cite{NENASHEV2017272} have different criteria for when the conjectured series is correct in degree $d+d'$, for $d'>1$.

In this paper, we restrict ourselves to the case $2 \leq d'<d$. The
coefficient of $t^{d+d'}$ in Fr\"oberg's conjecture in the equigenerated case then equals the coefficient of $t^{d+d'}$ in  
\[
\left[ \frac{1-rt^d}{(1-t)^n} \right],
\]
that is, is equal to \[\max ( \dim_k R_{d+d'}-r \dim_k R_{d'}, 0).\]
 
This means that $\{f_i m_j\}$ is linearly independent or spans  $R_{d+d'}$, where $\{m_j\}$ is a basis for $R_{d'}.$

Our main result is that the Fr\"oberg conjecture is true in the second non-trivial degree, that is, that the coefficient of $t^{d+2}$ is the conjectured one. 

\begin{theorem} \label{thm:d2} Let $k$ be an infinite field and let $d>2$. Then, for any $r$, there is an ideal $I$ generated by $r$ forms of degree $d$ such that $k[x_1,\dots,x_n]/I$ has the Hilbert series conjectured by Fr\"oberg up to degree $d+2$, that is, the Fr\"oberg conjecture is true up 
to degree $d+2$. 
\end{theorem}
Recall that to verify the conjecture for a fixed set of data, it suffices to check one example, as the expected behavior is an open condition.

In particular, Theorem \ref{thm:d2} gives that we know that the Fr\"oberg conjecture holds when $d>2$ and $r\dim_k R_d\ge \dim_k R_{d+2}$, i.e., when 
\[
r\ge \frac{\dim_k R_{d+2}}{\dim_k R_d} = \frac{(n+1+d)(n+d)}{(d+2)(d+1)}.
\]
We are also able to give the following asymptotic result.  
\begin{theorem} \label{thm:asymto}
    Let $k$ be an infinite field, and fix $d>1$. Then there exists an $N$ such that for any $n \ge N$, and any $r$, there is an ideal $I$ generated by $r$ forms of degree $d$ such that $k[x_1,\dots,x_n]/I$ has the Hilbert series conjectured by Fr\"oberg up to degree $2d-1$, that is, the Fr\"oberg conjecture is true up 
to degree $2d-1$.
\end{theorem}

The paper is organized as follows. In Section \ref{sec:aubry}, we first adapt the geometric argument used by Hochster and Laksov \cite{Hochster01011987} for the linear case to our case $d' \geq 2$, an adaption which is closely related to Aubry's  \cite{Aubry1995392}. Next, we use Eliahou's condensed version of Macaulay's theorem  \cite{Eliahou} to reduce our problem to the study of a univariate polynomial $g_{d,d'}$. 
In Section \ref{sec:lemmas} we study the polynomial $g_{d,d'}$ and show that if $\dim_k R_{d+d'} \geq \dim_k (R_{d'})^2$, then the Fr\"oberg conjecture holds up to degree $d+d'$ provided that  $g_{d,d'}$  has at most one sign-change in its coefficients.
In Section \ref{sec:quadrics} we give the proof of Theorem \ref{thm:d2},  where the results of Section \ref{sec:lemmas} are used to handle all but a finite number of cases. The remaining cases are proved by computer calculations.
In Section \ref{sec:higher} we show that $g_{d,d-1}$ has exactly one sign change and use this result to prove Theorem \ref{thm:asymto}. We also show that if $d' \leq \compBound$ and $\dim_k R_{d+d'} \geq \dim_k (R_{d'})^2$, then the Fr\"oberg conjecture holds up to degree $d+d'$.
In Section \ref{sec:discussions} we put our result in context with the existing literature and discuss both generalizations and limitations of our method.

\section{The Hochster-Laksov argument revisited} \label{sec:aubry}
Throughout the paper, we let $k$ denote an infinite field and $R = k[x_1,\ldots,x_n]$ the standard graded polynomial ring.

After choosing a $k$-basis for $R_d$, 
we  identify any $f\in R_d$
by its coefficients. 
Hence, we identify the space $R_d^r$
of $r$-tuples of forms of degree $d$, with the
space $\mathbb{A}^{r \times \dim_k R_d}$ as a space of $r \times \dim_k R_d$-matrices. As we interpret any element of $A \in \mathbb{A}^{r \times \dim R_d}$ as a matrix, we will refer to the rank of this matrix  as the rank of $A$.

Using the standard $k$-vector space structure on $\mathbb{A}^{r \times \dim R_d}$, the above interpretation is formally an isomorphism of vector spaces. Any vector subspace of $\mathbb{A}^{r \times \dim_k R_d}$ is an algebraic subset, and the topological (Krull) dimension agrees with the vector space dimension. Hence, we may speak of the dimension, which we denote $\dim_k$, of any subset of $\mathbb{A}^{r \times \dim_k}$ without ambiguity. 

Let $d$ and $d'<d$ be positive integers and let $\{m_i\}$ be a basis for $R_{d'}$. Recall that our main objective is to show that 
$ \{ f_i m_j \}$ is linearly independent or spans $R_{d+d'}$. In fact, we will do this by showing the stronger Property (\ref{gcase}).

Throughout the paper, let
$r=r(n,d,d')$ be the smallest integer such that 
\[r \dim_k R_{d'} \ge \dim_k R_{d+d'}\] and let \[s =s(n,d,d') =\dim_k R_{d+d'}- (r-1)\dim_k R_{d'}.\] 

Observe that $r-1$ and $s$ are essentially the quotient and remainder in the division of $\dim_k R_{d+'d}$ by $\dim_k R_{d'}$ and we have that  $1\leq s\leq \dim_k R_{d'}$. We will suppress the dependence on $n$, $d$ and $d'$ when they will be fixed. We would like to show that there exists $f_1, \dots, f_r \in R_d$ such that
\begin{equation} \label{gcase}
    \{f_im_j\}_{i<r} \cup \{f_rm_1, \dots, f_rm_s  \} \text{ is a $k$-basis for $R_{d+d'}$.}
\end{equation}

With our choice of basis $\{ m_i\}$ for $R_{d'}$, we identify $R_{d'}^r$ with $\mathbb A^{r\times \dim_k R_{d'}}$. We let $P$ be the linear subspace of the projective space $\mathbb P(R_{d'}^r)$ given by 
$$
P = \{p \in \mathbb{P}\left(R_{d'}^r\right) \mid p_{r,j}=0 \ \text{for} \ j> s \},
$$ which can be interpreted as the equivalence classes of $r$-tuples $(a_1,\dots,a_r)=\left(\sum p_{1,i}m_i,\dots,\sum p_{r,i}m_i\right)$ in $R_{d'}^{r}$ satisfying $a_r \in \text{span}\{m_1,\dots,m_s\}$, under projective equivalence. 

For any $F$ in $\mathbb A^{r\times \dim_k R_d}$ we will denote the corresponding $r$-tuple in $R_d^{r}$ by $(f_1,\dots,f_r)$, and for $A$ in $P$ we denote the corresponding (up to a non-zero constant) element $(a_1,\dots,a_r) \in R_{d'}^{r}$.

The main object we will study is the subvariety 
of $\mathbb A^{r\times \dim_k R_d}\times P$ defined as
\[
V = \{(F,A) \in \mathbb A^{r\times \dim_k R_d}\times P\mid  f_1a_1+ \dots + f_ra_r =0\}
\subseteq R_d^r\times \mathbb P\left(R_{d'}^r\right).
\] 
We indicate the dependence on $n,d,$ and $d'$, by writing $V(n,d,d')$, when needed.
The variety $V$ should be thought of as the locus where Property \eqref{gcase} does not hold. Indeed, $F=(f_1,\dots,f_r)$ lies in the image of the composition
$$
V \xhookrightarrow{} \mathbb A^{r\times \dim_k R_d}\times P \twoheadrightarrow \mathbb A^{r\times \dim_k R_d}
$$ if and only if there exists some $(a_1,\dots,a_r)$ in $P$ such that $f_1a_1+\dots+f_ra_r = 0$, which since $a_r \in \text{span}\{m_1,\dots,m_s\}$, shows that Property \eqref{gcase} does not hold. 

\begin{lemma}\label{lemma: dimension-ineq-implies-Fröberg}
Fix $d > d' > 0.$ Suppose that $\dim_k V < r\dim_k R_{d}$. Then, for any $r'$, there exists an ideal $(f_1,\dots,f_{r'}) \subset R$ with $\deg f_i = d$ that has the Hilbert function conjectured by Fr\"oberg in degree $d+d'$.
\end{lemma}
\begin{proof} 
    If $\dim_k V < r\dim_k R_d$, then the projection onto the first coordinate, $V \to \mathbb A^{r\times \dim_k R_d}$, cannot be surjective. Hence there exists at least one set of forms $F = \{f_1,\dots,f_r\}$ such that $  \{f_im_j\}_{i<r} \cup \{f_rm_1, \dots, f_rm_s  \}$ is a basis for $R_{d+d'}$, which proves the statement for the case $r'=r.$

    Since $\{f_i m_j\}_{i<r}$ is linearly independent, this gives the claim for the case $r'<r$.

    The case $r'>r$ follows by the fact that if the ideal generated by $F$ spans $R_{d+d'}$ in degree $d+d'$, then so does the ideal generated by $F \cup \{f_{r+1},\ldots,f_{r'}\}.$
\end{proof}

We devote the rest of this section to try to understand the dimension of $V$. Following \cite{Hochster01011987}, we use the rank to partition the projective space $P$.

Let $P_t$ denote the locally closed subset of $P$ consisting of elements whose corresponding $r\times \dim_k R_{d'}$-matrix has rank $\dim_k R_{d'} -t$. Furthermore, we let $Y_t$ denote the locally closed subset of $P$ of elements having rank $\dim_k R_{d'}-t$ and having last row equal to zero. We also let $X_{t,h}$ denote the locally closed subset of $P$ represented by matrices whose last row is not zero, having rank $\dim_k R_{d'}-t$, and such that the last $\dim_k R_{d'}-s$ columns have rank $h$. Note that $h$ can take the values $0\le h \le \min\{ \dim_k R_{d'}-t-1, \dim_k R_{d'}-s\}$.
\begin{remark}\label{rmk:rankdim}
The entries of any $A \in P_t$ correspond to the coefficients with respect to a basis for $R_{d'}$ of the corresponding $(a_1,\dots,a_r) \in R_{d'}^{r}$. Hence, when viewing $(a_1,\dots,a_r)$ as and ideal in $R$, we have $\dim_k ((a_1\dots,a_r)_{d'}) = \dim_k R_{d'} -t$, and thus $\dim_k \left( (R/(a_1,\dots,a_r))_{d'} \right)=t$
\end{remark}

We now recall the following proposition. Note that the statements here are adjusted for $r \times \dim_k R_{d'}$-matrices, whereas \cite{Hochster01011987} only consider the case $d'=1$. 
\begin{prop}[{\cite[Proposition~2]{Hochster01011987}}] \label{prop: dimensions partition P} Given $d' >0,$ assume that $r>\dim_k R_{d'}$. Then the following  hold.
\begin{enumerate}[label=\alph*)]
    \item For each $t$, $0\le t \le (\dim_k R_{d'} -1)$, we have that $Y_t $ is irreducible, and
    $$
    \dim_k Y_t = (r+t-1)(\dim_k R_{d'}-t) -1.
    $$
    \item For each $t,h$ with $0\le t \le (\dim_k R_{d'}-1)$, $1\le h \le \min\{\dim_k R_{d'}-t-1,\dim_k R_{d'}-s  \}$, we have that $X_{t,h}$ is irreducible of dimension
    $$
    (r+s+t-\dim_k R_{d'})(\dim_k R_{d'} -t) - 1 + (2\dim_k R_{d'} -s-t-1)h - h^2.
    $$
    \item Let $m =\min\{\dim_k R_{d'}-t-1,\dim_k R_{d'}-s  \}.$ For a fixed $t$, $\dim_k X_{t,h}$ is maximized when $h = m$. Moreover, when $m\neq \dim_k R_{d'}-s$, it holds that $\dim_k X_{t,h} \le \dim_k Y_t$.
\end{enumerate}
\end{prop}

We can now prove the following result, which is a variation of 
\cite[Corollary~3]{Hochster01011987}.

\begin{prop}\label{prop: dimension P_t}
Given $d>d' > 0,$ assume that $r>\dim_k R_{d'}$. For $t$, $0\le t\le (\dim_k R_{d'}-1)$, we have
$$
\dim_k P_t = \dim_k R_{d+d'} + t(\dim_k R_{d'} - t-r) +\max\{t-s,0\}-1.
$$
\end{prop}
\begin{proof}
Since $P_t = Y_t \cup \left( \ \bigcup_h X_{t,h}\ \right )$, it follows from Proposition \ref{prop: dimensions partition P}c 
that $$\dim_k P_t = \max\{\dim_k Y_t,\max_h\{\dim_k X_{t,h}\} \} = \max\{ \dim_k Y_t, \dim_k X_{t,m(t)} \},$$ where $m(t) = \min\{\dim_k R_{d'}-t-1,\dim_k R_{d'}-s  \}$. If $t+1\ge s$ we have that $m(t) \neq \dim_k R_{d'}-s$, and hence by Proposition \ref{prop: dimensions partition P}c that $\dim_k X_{t,m(t)} \le \dim_k Y_t$. Using that $\dim_k R_{d+d'} = (r-1)\dim_k R_{d'}  + s $ we obtain 
$$
\dim_k P_t = \dim_k Y_t = \dim_k R_{d+d'} +(t-s) +  t(\dim_k R_d' - t-r) -1.
$$ If instead $t \le s$, then $m(t) =\dim_k R_{d'}-s $, and inserting $h=\dim_k R_d'-s$ into Proposition \ref{prop: dimensions partition P}b
gives 
$$
\dim_k X_{t,(\dim_k R_{d'}-s)} = \dim_k R_{d+d'} +  t(\dim_k R_d' - t-r) -1,
$$ which concludes the proof.
\end{proof}

We now recall the condensed version of Macaulay's theorem. Following \cite{Eliahou}, for $x \in \mathbb{R}$ and $i$ a non-negative integer, we define 
\[\binom{x}{i} = 
\frac{x(x-1) \cdots (x-i+1)}{i!} \] if $i \geq 1$, and $\binom{x}{0} = 1$.

\begin{theorem}[{\cite[Theorem 5.10]{Eliahou}}] \label{thm:Eliahou}
Let $R$ 
be a standard graded $k$-algebra.
Let $j \geq 1$ be an integer and let 
$x \geq j-1$ be the unique real number satisfying $\dim_k R_j = \binom{x}{j}$. Then 

\[ 
\dim_k R_{j-1} \geq \binom{x-1}{j-1} \text{ and } \dim_k R_{j+1} \leq \binom{x+1}{j+1}. 
\]

\end{theorem}

We now introduce the polynomial $g_{d,d'}(x)$, which will be our main object of study in the following sections. For fixed $d'$ and $d>d'$, we let
\begin{equation}
    g_{d,d'}(x) = \frac{\binom{x+d+d'}{d}}{\binom{d+d'}{d}} - \binom{x+d'}{d'}.
\end{equation}

\begin{lemma} \label{lemma:equiv-}
    For $d > d' > 0$, the following are equivalent:
    \begin{itemize}
        \item $g_{d,d'}(n-1) \ge 0$,
        \item $\dim_k R_{d+d'} \ge (\dim_k R_{d'})^2$,
\item $r \ge \dim_k R_{d'}$ if $s=\dim_k R_{d'}$, and $r>\dim_k R_{d'}$ if $s<\dim_k R_{d'}$.
    \end{itemize}
\end{lemma}
\begin{proof}
By definition we have
$$
 g_{d,d'}(n-1) = \frac{\binom{n-1+d+d'}{d}}{\binom{d+d'}{d}} - \binom{n-1+d'}{d'}  = \frac{\binom{n-1+d+d'}{d+d'}}{\binom{n-1+d'}{d'}} - \binom{n-1+d'}{d'}. 
$$ After multiplying with $\dim_k R_{d'} = \binom{n-1+d'}{d'}$ and using that $\dim_k R_{d+d'} = \binom{n-1+d+d'}{d+d'}$ we obtain
$$
\dim_k R_{d+d'} - (\dim_k R_{d'})^2.
$$ Thus the two first statements are equivalent.

If $s= \dim_k R_{d'}$, then $r= \dim_k R_{d+d'}/\dim_k R_{d'}$ and hence we clearly have that $r\ge \dim_k R_{d'} \iff \dim_k R_{d+d'} \ge (\dim_k R_{d'})^2$. Otherwise, we have $s<\dim_k R_{d'}$. Using $$
\dim_k R_{d+d'} = (r-1)\dim_k R_{d'} +s < r\dim_k R_{d'} 
$$ we find that $\dim_k R_{d+d'} \ge (\dim_k R_{d'})^2 $ implies that $r> \dim_k R_{d'}$. Conversely, if $r>\dim_k R_{d'}$, then 
$$
\dim_k R_{d+d'} = (r-1)\dim_k R_{d'} +s \ge (\dim_k R_{d'})^2.
$$
\end{proof}

\begin{theorem}\label{thm: f-inequality}
Fix $d > d'> 0$ and suppose that 
\[g_{d,d'}(x)\le g_{d,d'}(n-1) \text{ \,\,for all \,\,} x \in [0,n-1]. \] Then, for any $r$, there is an ideal $I$ generated by $r$ forms of degree $d$ such that $\dim_k (R/I)_{d+d'}$ equals the value conjectured by Fr\"oberg.

\end{theorem}
\begin{proof}
    Note that by assumption we have that $g_{d,d'}(n-1)\ge g(0)=0$, and hence by Lemma \ref{lemma:equiv-} that $r\ge \dim_k R_{d'}$. 
    Using Lemma \ref{lemma: dimension-ineq-implies-Fröberg} it is enough to show that $\dim_k V < r\dim_k R_d$,
    and since $V = \bigcup_{t} V_t$ it suffices to show $\dim_k V_t < r\dim_k R_d$ for each $0\le t\le \dim_k R_{d'}-1$. Consider the projection, $\pi:V_t\twoheadrightarrow P_t$, onto the second coordinate of $V_t$, i.e. the composition
    $$
    V_t \hookrightarrow \mathbb A^{r\times \dim_k R_d}\times P_t \twoheadrightarrow P_t.
    $$ Clearly this is surjective, as $(0,A) \in V$ for all $A \in P$. After restricting to a irreducible componenet of maximal dimension and applying the fiber dimension theorem we obtain that
$$ 
\dim_k V_t \le \dim_k P_t + \dim_k \pi^{-1}(A),
$$ for some $A \in P_t$.
     
    For any $A \in P_t$ the fiber $\pi^{-1}(A)$ is identified with its projection onto $\mathbb A^{r\times \dim_k R_d}$, which when identified as a subset of $R_d^r$ corresponds to the vector space of elements $(f_1,\dots,f_r)$ satisfying $f_1a_1 +\dots f_ra_r = 0$, where $(a_1,\dots,a_r)$ is a representative for the element corresponding to $A$ in $R_{d'}^r$. This is exactly the kernel of the map
    $$
    R_d^r \xrightarrow[]{\cdot (a_1,\dots,a_r)} R_{d+d'}.
    $$ Using the exact sequence
$$
0 \to \pi^{-1}(A) \to  R_d^r \xrightarrow[]{\cdot (a_1,\dots,a_r)} R_{d+d'} \to (R/(a_1,\dots,a_r))_{d+d'} \to 0,
$$ we find that 
$$
\dim_k \pi^{-1}(A) = r\dim_k R_d - \dim_k R_{d+d'} + \HF_{R/(a_1,\dots,a_r)}(d+d'),
$$
where $\HF$ denotes the Hilbert function.
Now by Remark \ref{rmk:rankdim} we have that $\HF_{R/(a_1,\dots,a_r)}(d')=t$. Let $x$ be the real number satisfying $t = \binom{x}{d'}$. Note that when $t=0$ then $x=d'-1$, and for any $t\ge 1$ then $x\in[d',n+d'-1].$ By repeated use of Theorem \ref{thm:Eliahou}, we have that $\HF_{R/(a_1,\dots,a_r)}(d+d') \le \binom{x+d}{d+d'}$. From Lemma \ref{lemma:equiv-} we have that either $r>\dim_k R_{d'}$, or $r=\dim_k R_{d'}$ and $s=\dim_k R_{d'}$. If $r > \dim_k R_{d'}$, then it follows from Proposition \ref{prop: dimension P_t} that
$$
\dim_k P_t = \dim_k R_{d+d'} + t(\dim_k R_{d'} - t-r) +\max\{t-s,0\}-1.
$$ If instead $r=\dim_k R_{d'}$ and $s=\dim_k R_{d'}$, then $P$ is just the projectivisation of the space of $r\times \dim_{k}R_{d'}$-matrices, and it is well known that
\begin{align*}
    \dim_k P_t = (r+t)(\dim_k R_{d'}-t)-1 \\=
    \dim_k R_{d+d'} + t(\dim_k R_{d'} - t-r) -1.
\end{align*} Since $t< \dim_k R_{d'}=s$, this coincides with the previous expression.

Inserting this into our upper bound on $\dim_k V_t$ we find that
\begin{align*}
    \dim_k V_t \le t(\dim_k R_{d'} - t-r) +\max\{t-s,0\}  + r\dim_k R_d  +\binom{x+d}{d+d'}-1.
\end{align*}
To prove that $\dim_k V_t < r\dim_k R_d$ it is then sufficient to show the inequality
$$
t(\dim_k R_{d'} - t-r) +\max\{t-s,0\}   +\binom{x+d}{d+d'}\le 0.
$$

Consider first the case when $t-s \le 0$. Then, using that $r \dim_k R_{d'} \ge \dim_k R_{d+d'}$ we can multiply the LHS above with $\dim_k R_{d'}$ to obtain the following upper bound
\begin{align} \label{eq:upbound}
-t\left( \dim_k R_{d+d'} - (\dim_k R_{d'})^2\right) + \dim_k R_{d'}\left( \binom{x+d}{d+d'}-t^2\right).
\end{align}
After dividing by $t$ and $\dim_k R_{d'}$ and inserting $t = \binom{x}{d'}$ and $\dim_k R_y = \binom{n+y-1}{y}$ this equals
$$
-\frac{\binom{n+d+d'-1}{d+d'}}{\binom{n+d'-1}{d'}} + \binom{n+d'-1}{d'} + \frac{\binom{x+d}{d+d'}}{\binom{x}{d'}} - \binom{x}{d'}.
$$
We find, using that $\frac{\binom{x+d}{d+d'}}{\binom{x}{d'}} = \frac{\binom{x+d}{d}}{\binom{d+d'}{d}}$, that this is less than or equal to zero if and only if
$$
\frac{\binom{x+d}{d}}{\binom{d+d'}{d}}-\binom{x}{d'} \le \frac{\binom{n+d+d'-1}{d}}{\binom{d+d'}{d}}-\binom{n+d'-1}{d'},
$$ where $x=d'-1$ or $x\in [d',n+d'-1]$. Notice that the left hand side vanishes when $x=d'-1$ and when $x=d'$. Hence, it suffices to consider $x \in [d',n-1+d']$, which after shifting $x$ by $d'$ corresponds to the condition $g_{d,d'}(x) \le g_{d,d'}(n-1)$ for $x \in [0,n-1]$.

In the case where $t-s > 0$, we use the fact that $r\dim_k R_{d'} = \dim_k R_{d'} + \dim_k R_{d+d'} -s$ to get
\begin{align*}
    \dim_k R_{d'}\left( t(\dim_k R_{d'} - t-r) +(t-s)   +\binom{x+d}{d+d'}\right) = \\
    -t(\dim_k R_{d+d'} - (\dim_k R_{d'})^2) + \dim_k R_{d'}\left( \binom{x+d}{d+d'}-t^2\right) + s(t-\dim_k R_{d'}).
\end{align*}
Since $t \le \dim_k R_{d'}$, the expression \eqref{eq:upbound} from the previous case is an upper bound, which finishes the proof.
\end{proof}

The problem of showing that $\dim_k V < r \dim_k R_d$ has thus been reduced to showing an inequality of an explicit polynomial. We believe that $g_{d,d'}(x) \le g_{d,d'}(n-1)$ always holds, under the necessary assumption that $g_{d,d'}(n-1)\ge0$, but we have not been able to prove it in general. The cases where we have been able to prove it are covered in Section \ref{sec:quadrics} and Section \ref{sec:higher}, but first we need to better understand the polynomial $g_{d,d'}$.

\section{A sufficient condition for the Fr\"oberg conjecture to be true up to degree $d+d'$} \label{sec:lemmas} 

We now give a sufficient condition for the Fr\"oberg conjecture to be true up to degree $d+d'$ is that $\dim_k R_{d+d'} \geq (\dim_k R_{d'})^2$ and that
$g_{d,d'}(x)$ has at most one sign-change in its coefficients. For this purpose we will need the following lemmas.

\begin{lemma}\label{lemma: coefficients}
    The coefficient of $x^j$ in $g_{d,d'}(x)$ is given by
\[
\sum_{d'+1 \le b_1 < \dots < b_j \le d+d'} \frac{1}{b_1\cdots b_j} - \sum_{1\le b_1 < \dots < b_j \le d'} \frac{1}{b_1\cdots b_j}.
\]
\end{lemma}
\begin{proof}
By definition we have $g_{d,d'}(x) =\frac{\binom{x+d+d'}{d}}{\binom{d+d'}{d}} - \binom{x+d'}{d'}$. The two terms can be written as
$$
\frac{\binom{x+d+d'}{d}}{\binom{d+d'}{d}} = \frac{(x+d+d')\cdots(x+d'+1)}{(d+d')\cdots(d'+1)} = \prod_{i=d'+1}^{d+d'}\left(1+\frac{x}{i}\right)
$$ 
and 
$$
\binom{x+d'}{d'} = \frac{(x+d')\cdots(x+1)}{d'(d'-1)\cdots 1}
 = \prod_{i=1}^{d'}\left(1+\frac{x}{i}\right).
$$
The coefficients of $x^j$ in these terms are given by the elementary symmetric polynomial of degree $j$ evaluated at $(1/(d'+1),1/(d'+2),\dots,1/(d'+d))$ and $(1/1,1/2,\dots,1/d')$, respectively, and the statement follows.
\end{proof}
\begin{remark}\label{remark: increase d, decrease d'}
From Lemma \ref{lemma: coefficients} we see that the coefficient of $x^j$ in $g_{d,d'}(x)$ is increasing with respect to $d$, and decreasing with respect to $d'$.
\end{remark}
\begin{lemma}\label{lemma: degree-increase1}
Fix $d>d'>0$, and suppose that $g_{d,d'}(n-1)\ge g_{d,d'}(x)$ for all $x \in [0,n-1]$. Then $g_{d+1,d'}(n-1)\ge g_{d+1,d'}(x)$ for all $x \in [0,n-1]$.
\end{lemma}

\begin{proof}
We have that 
$$
g_{d,d'}(n-1) - g_{d,d'}(x) = \frac{\binom{n+d+d'-1}{d'}}{\binom{d+d'}{d}} - \binom{n+d'-1}{d'} -\frac{\binom{x+d+d'}{d}}{\binom{d+d'}{d}} + \binom{x+d'}{d'},
$$ and the only terms depending on $d$ are \[\frac{\binom{n+d+d'-1}{d}}{\binom{d+d'}{d}} = \frac{\binom{n+d+d'-1}{d+d'}}{\binom{n+d'-1}{d'}}, \text{ and } \frac{\binom{x+d+d'}{d}}{\binom{d+d'}{d}}=\frac{\binom{x+d+d'}{d+d'}}{\binom{x+d'}{d'}}.\]
By assumption, we have that $g_{d,d'}(n-1) \ge g_{d,d'}(x)$, for all $x \in [0,n-1]$, or equivalently
$$
\frac{\binom{n+d+d'-1}{d+d'}}{\binom{n+d'-1}{d'}} - \frac{\binom{x+d+d'}{d+d'}}{\binom{x+d'}{d'}} \ge  \binom{n+d'-1}{d'} -  \binom{x+d'}{d'}. 
$$ If we increase $d$ by one the LHS satisfies
\begin{align*}
 \frac{\binom{n+d+d'}{d+d'+1}}{\binom{n+d'-1}{d'}} -\frac{\binom{x+d+1+d'}{d+d'+1}}{\binom{x+d'}{d'}} = \frac{n+d+d'}{d+d'+1}\frac{\binom{n+d+d'-1}{d+d'}}{\binom{n+d'-1}{d'}} -  \frac{x+1+d+d'}{d+d'+1}\frac{\binom{x+d+d'}{d+d'+1}}{\binom{x+d'}{d'}}   \\
 \ge \frac{n+d+d'}{d+d'+1}\left(\frac{\binom{n+d+d'-1}{d+d'}}{\binom{n+d'-1}{d'}} - \frac{\binom{x+d+d'}{d+d'}}{\binom{x+d'}{d'}}  \right) \ge \binom{n+d'-1}{d'} -  \binom{x+d'}{d'}.
\end{align*}

\end{proof}

\begin{lemma}\label{lemma: degree-increase2}
Suppose that $n(d-d') \ge (d')^2$. Then $g_{d,d'}(n) \ge g_{d,d'}(n-1).$
\end{lemma} 

\begin{proof}
We have that
$$
g_{d,d'}(n) = \frac{\binom{n+d+d'}{d}}{\binom{d+d'}{d}} - \binom{n+d'}{d'} = \frac{n+d+d'}{n+d'}\frac{\binom{n+d+d'-1}{d'}}{\binom{d+d'}{d}} - \frac{n+d'}{n}\binom{n+d'-1}{d'}.
$$
Since 
$$\frac{n+d+d'}{n+d'} \ge \frac{n+d'}{n} \iff n(d-d') \ge (d')^2,
$$ we find that the assumption $n(d-d') \ge (d')^2$ implies that
\begin{align*}
g_{d,d'}(n) \ge \frac{n+d+d'}{n+d'}\left(\frac{\binom{n+d+d'-1}{d'}}{\binom{d+d'}{d}} - \binom{n+d'-1}{d'} \right) = \frac{n+d+d'}{n+d'}g_{d,d'}(n-1) \\\ge g_{d,d'}(n-1).
\end{align*}
\end{proof}

\begin{lemma}\label{lemma: sign-change}
    Fix $d>d'>0$, and suppose that $g_{d,d'}$ has at most one sign-change in its coefficients. Then $g_{d,d'}(x) \le \max\{0,g_{d,d'}(n-1)\}$ for all $x\in [0,n-1]$.
\end{lemma}

\begin{proof} Since $g_{d,d'}(0) = 0$ we have that if $g_{d,d'}$ has at most sign-change in its coefficients, then so does $\frac{d}{dx}g_{d,d'}$. Thus, by Descartes' rule of signs, $\frac{d}{dx}g_{d,d'}$ has exactly one strictly positive real root, and so $g_{d,d'}$ has exactly one stationary point $x_0$, and therefore, the sign of $\frac{d}{dx}g_{d,d'}$ is constant for $x\ge x_0$. Since $g_{d,d'}$ is of degree $d$ with a positive coefficient of $x^d$, we find that $\frac{d}{dx}g_{d,d'}$ must eventually be strictly positive, and in particular $\frac{d}{dx}g_{d,d'}(y) \ge 0$ for $y\ge x_0$. Thus $x_0$ is not a maximum, and on any closed interval $g_{d,d'}$ attains its maximum value at the boundary, which proves the claim.
\end{proof}

\begin{theorem}\label{thm: sign-changes} Fix $d>d'>0$, and suppose that $g_{d,d'}(x)$ has at most one sign-change in its coefficients. Then for any $n$ such that $\dim_{k}R_{d+d'} \ge (\dim_k R_{d'})^2$ we have that, for any $r$, there is an ideal $I$ generated by $r$ forms of degree $d$ such that  $\dim_k(R/I)_{d+d'}$ equals the value conjectured by Fr\"oberg.
\end{theorem}
\begin{proof}
    From Lemma \ref{lemma: sign-change} it follows directly that $g_{d,d'}(x)\le \max\{0, g_{d,d'}(n-1)\}$ for all $x \in [0,n-1]$. The assumption $\dim_{k}R_{d+d'} \ge (\dim_k R_{d'})^2$ is equivalent to $g_{d,d'}(n-1)\ge0$, by Lemma \ref{lemma:equiv-}, and hence $g_{d,d'}(x)\le g_{d,d'}(n-1)$ for all $x\in [0,n-1]$. Theorem \ref{thm: f-inequality} finishes the proof.
\end{proof}

\section{Multiplication by quadrics} \label{sec:quadrics}

In this section we consider the case $d'=2$, which corresponds to investigating the Hilbert function in degree $d+2$ of a generic ideal generated by forms of degree $d$. In Section \ref{sec:aubry} we showed, under some assumptions, that an inequality involving the function $g_{d,d'}$, is sufficient to obtain a degree-wise result on Fr\"oberg's conjecture. When $d'=2$, we can show that this inequality does hold, whenever possible.

\begin{prop}\label{prop: sign-change quadratic and linear}
Fix $d>d'>0$, and suppose that the coefficient of $x$ in $g_{d,d'}(x)$ is non-negative. Then the coefficient of $x^2$ is also non-negative.
\end{prop}
\begin{proof}
    By Lemma \ref{lemma: coefficients}, we have that
$$
\frac{\binom{x+d+d'}{d}}{\binom{d+d'}{2}} = 1 + \left(\sum_{i=d'+1}^{d+d'}\frac{1}{i} \right)x  + \left(\sum_{\substack{d'+1\le i<j\le d+d' }} \frac{1}{ij}\right) x^2 + \cdots, 
$$
and $$
\binom{x+d'}{d'} = 1 + \left(\sum_{i=1}^{d'}\frac{1}{i} \right)x+\left(\sum_{\substack{1\le i<j\le d' }} \frac{1}{ij}\right) x^2 + \cdots.
$$
By assumption the coefficient of $x$ is non-negative, that is
\[
 \sum_{i=d'+1}^{d+d'}\frac{1}{i}  \ge \sum_{i=1}^{d'}\frac{1}{i}.
\]
Squaring both sides now gives the inequality 

\[
 \sum_{i=d'+1}^{d+d'}\frac{1}{i^2} + 2\sum_{\substack{d'+1\le i<j\le d+d'}}\frac{1}{ij}   \ge \sum_{i=1}^{d'}\frac{1}{i^2} + 2\sum_{\substack{1\le i<j\le d'}}\frac{1}{ij}.
\]

Since 
\[
\sum_{i=1}^{\infty}\frac{1}{i^2} = \frac{\pi^2}{6} <2 \text{ \, \, and \, \, } 
\sum_{i=1}^{d'}\frac{1}{i^2} \geq 1,
\] we find that
$$
\sum_{i=d'+1}^{d+d'}\frac{1}{i^2} < \sum_{i=1}^{d'}\frac{1}{i^2},
$$ and hence
\[
\sum_{\substack{d'+1\le i<j\le d+d'}}\frac{1}{ij} > \sum_{\substack{1\le i<j\le d'}}\frac{1}{ij}.
\]
\end{proof}

\begin{corollary} \label{cor:d2}
    There is at most one sign change in the coefficients of $g_{d,2}$, for $d>2$.
\end{corollary}
\begin{proof} 
 Since $$
 \binom{x+2}{2} = 1 + \frac{3}{2}x +\frac{1}{2}x^2,
 $$ we get that all coefficients of $g_{d,2}$, except possibly for the coefficients of $x$ and $x^2$, are non-negative. If the coefficient of $x$ in $g_{d,2}$ is non-negative, then, according to  
  Proposition \ref{prop: sign-change quadratic and linear}, then so is the coefficient of $x^2$. Hence there is at most one sign change among the coefficients of $g_{d,2}$.
\end{proof}

\begin{lemma} \label{lemma: d'=2-cases}
The inequality  $g_{d,2}(n-1)\ge 0$ holds in the cases
\begin{itemize}
    \item $d = 3$ and $n \ge 22$;
     \item $d=4$ and $n\ge 6$;
    \item $d\ge 5$ and $n\ge 3$.
\end{itemize}
\end{lemma}
\begin{proof}
    A calculation shows that $g_{3,2}(21) = 7,  g_{4,2}(5) = 1$, 
    and $g_{5,2}(2) = 0$. Since each of these cases satisfies $n(d-2) \ge 4$, the claim follows by Lemma~\ref{lemma: degree-increase2}.
\end{proof}

To handle the cases not covered by Lemma \ref{lemma: d'=2-cases}, the following lemma will be useful. 

\begin{lemma} \label{lemma:computations} 
Fix $d>d'>0$ and let $n'<n$ and let $R'=k[x_1,\ldots,x_{n'}]\subseteq R$.
Let 
$$\{m_i\}_{i=1}^{\dim_k R_{d'}} = \{m_i'\}_{i=1}^{\dim_k R_{d'}'} \cup \{m_i''\}_{i=1}^{\dim_k R_{d'}-\dim_k R_{d'}'}$$
be a basis for $R_{d'}$, such that $m_i' \in R'_{d'}$ and $m_i'' \notin R'_{d'}.$

Suppose that there exist forms $p_1,\dots,p_l$ in $R_d'$ such that 
\begin{equation}\label{eq: cond1}
    \{p_im_j'\} \text{ is a $k$-basis for $R_{d+d'}'$}.
\end{equation}
Suppose also that there exist forms $f_1,\dots,f_r$ with $f_i\in R_d'$ for $i=1,\ldots, l$ such that 
\begin{equation} \label{eq: cond2}
\{f_im_j''\}_{i\le l} \ \cup \{f_im_j\}_{l<i<r} \cup \ \{f_rm_1,\dots, f_rm_s\}
\text{ is a $k$-basis for $R_{d+d'}/R_{d+d'}'$.}
\end{equation}
 Then there exist forms $h_1,\dots,h_r$, such that
$$
 \{h_im_j\}_{i<r} \cup \ \{h_rm_1,\dots, h_rm_s\} \text{ is a $k$-basis for $R_{d+d'}$}
$$
\end{lemma}
\begin{proof} 
 We remark first that Condition \eqref{eq: cond1} implies that the number of forms $l$ of degree $d$ equals   
$\dim_k R_{d+d'}' / \dim_k R_{d'}',$ and especially, that $n'$ must be chosen so that $\dim_k R_{d'}'$ divides  $\dim_k R_{d+d'}'$.

 Any set $\{f_1,\dots,f_l\}$ with $f_i \in R_d'$ corresponds to a point in $\mathbb{A}^{ l\times \dim_k R_{d}'}$. Condition \eqref{eq: cond1} is an open condition, so there exists a non-empty subset of $\mathbb{A}^{ l\times \dim_k R_{d}'} $ consisting of collections of $l$ forms in $R_d'$ with this property. 
 
Condition \eqref{eq: cond2}
is similarly open on  $\mathbb{A}^{l\times \dim_k R_{d}'} \times \mathbb{A}^{(r-l)\times \dim_k R_{d}}$. Hence there exists a non-empty open subset of $\mathbb{A}^{l\times \dim_k R_{d}'} \times \mathbb{A}^{(r-l)\times \dim_k R_{d}'}$ corresponding to collections of forms $h_1,\dots,h_r$ such that $h_1,\dots,h_l$ satisfy \eqref{eq: cond1} and such that $h_{1},\dots,h_r$ satisfy \eqref{eq: cond2}.

It follows that
$$
 \{h_im_j\}_{i<r} \cup \ \{h_rm_1,\dots, h_rm_s\} \text{ is a $k$-basis for $R_{d+d'}$.}
$$
\end{proof}

\begin{prop} \label{prop:calc}
The Fr\"oberg conjecture is true up to degree $d+2$ in the cases
\begin{itemize}
    \item $d=3$ and $n \le 21$;
     \item $d=4$ and $n \le 5$.
\end{itemize}
\end{prop}

\begin{proof}
For the cases $d=4, n \leq 5$ and $d=3, n \leq 15$ we can perform a standard calculation in Macaulay2 \cite{M2} to check the correctness of the statement. For higher values of $n$ we use Lemma \ref{lemma:computations} and Macaulay2 \cite{M2}. All computations are performed over a finite field for efficiency reasons. This is motivated by the fact that a relation over $\mathbb{Q}$ can be written as a relation over $\mathbb{Z}$ by clearing the denominators. After dividing out common factors, such a relation holds non-trivially also in a finite field with $p$ elements.

The choices of $n'$ and $p$ are given in Table \ref{fig:table}.
The Macaulay2 code used for the calculations can be found \href{https://github.com/Elrokaren/Computer-calculations---Independence-of-generic-forms-and-the-Fr-berg-conjecture.}{here}.

\begin{table}[h!]
\centering
\begin{tabular}{c|ccccccccc}
\hline
$n$  &$\leq 15$ & 16 & 17 & 18 & 19 & 20 & 21 \\
$n'$ &- & 13 & 13 & 16 & 17 & 18 & 18 \\
$l$  & - & 68 & 68 & 114 & 133 & 154 & 154 \\
$p$ & 11 & 11 & 11 & 11 & 5 & 5 & 5 \\
\hline
\end{tabular}
\caption{Choices for $n'$ and $p$ in Lemma \ref{lemma:computations} for different values of $n$. The value $l$ is the quotient 
$\dim_k R_{5}'/\dim_k R_2'.$} \label{fig:table}
\end{table}

\end{proof}

We are now ready to prove Theorem \ref{thm:d2}.

\begin{proof}[Proof of Theorem \ref{thm:d2}]
Suppose first that $d=3$, $n \geq 22$, or $d=4$, $n \geq 6$, or $d\geq 5$, $n \geq 3$. By Corollary \ref{cor:d2}, there is only one sign change among the coefficients of $g_{d,2}$. Moreover, from Lemma \ref{lemma: sign-change} and Lemma \ref{lemma: d'=2-cases}, 
 it follows that
$g_{d,2}(x) \leq g_{d,2}(n-1)$ for all 
$x \in [0,n-1]$. Now Theorem \ref{thm: f-inequality} gives us the claim.

Since the Fr\"oberg conjecture is true for $n \leq 3$, see \cite{ANICK1986235,Froberg1985}, we are left with the cases $d=4,n \leq  5$ and $d=3, n \leq 21$, which are covered by Proposition \ref{prop:calc}.

\end{proof}

\section{The case $d' \geq 3$} \label{sec:higher}
Having established that the Fr\"oberg conjecture is true up to degree $d+2$ for $d > 2$, we now turn to the case $d' \geq 3.$ We give a result on the sign changes of $g_{d,d-1}(x)$, and used it to prove Theorem \ref{thm:asymto}. We also show that for any $d'\leq \compBound$ and $d > d'$, the Fr\"oberg conjecture is true up to degree $d+d'$ provided that 
$\dim_k R_{d+d'} \geq (\dim_k R_{d'})^2.$  

\begin{prop}\label{prop: negative-coeff-d-1}
    The coefficient of $x^{d}$ in $g_{d,d-1}(x)$ is positive, while all other coefficients 
    are negative, for all $d\geq2$.
\end{prop}
\begin{proof}
For technical reasons, we let $d'=d-1$ throughout the proof. The coefficient of $x^{d'+1}$ in $g_{d'+1,d'}(x)$ is positive by construction.

Let $j \leq d'+1.$ By Lemma \ref{lemma: coefficients}, the coefficient of $x^j$ in  $g_{d'+1,d'}(x)$ is  equal to
\[
\sum_{d'+1 \le b_1 < \dots < b_j \le 2d'+1} \frac{1}{b_1\cdots b_j} - \sum_{1\le b_1 < \dots < b_j \le d'} \frac{1}{b_1\cdots b_j}.
\]

We prove that this is negative for $j \leq d'$ by first considering $j=d'$, thereafter $j=1$ by induction on $d'$, and lastly, the remaining terms, also by induction on $d'$.
For $j=d'$, the coefficient is

\[
\begin{gathered}
    \frac{d'+1 + \dots + 2d'+1}{(d'+1)\cdots (2d'+1)} - \frac{1}{d'!} \le \frac{(2d'+1)(d'+1)}{(d'+1)\cdots (2d'+1)} - \frac{1}{d'!} \le \\
    \frac{1}{2 d' (d'+2)\cdots (2d'-1)} - \frac{1}{d'!} < 0,
\end{gathered}
\] where the last inequality is given by observing that both terms contain $d'$ factors and each factor in the left term is smaller than the corresponding factor in the right term. 

Continuing with the coefficient of $x$ we argue by induction on $d'$, with the case $d'=1$ following by the previous argument. The coefficient is given by
\[
\sum_{d'+1 \le b\le 2d'+1} \frac{1}{b} - \sum_{1 \le b\le d'} \frac{1}{b}.
\] From the induction assumption we already know that
\[
\sum_{d' \le b\le 2d'-1} \frac{1}{b} < \sum_{1 \le b\le d'-1} \frac{1}{b}
\] and since $\frac{1}{d'} > \frac{1}{2d'}+ \frac{1}{2d'+1}$, it follows that the coefficient is negative.

For the coefficients of $x^j$, $1<j<d'$, we argue by induction on $d'$, assuming that  coefficients are negative for any $j$ for $d'-1$.
Rewriting the expressions as
\begin{align*}
  \sum_{d'+1 \le a_1 < \dots < a_j \le 2d'+1} \frac{1}{a_1\cdots a_j} = \sum_{d'+1 \le a_1 < \dots < a_j \le 2d'} \frac{1}{a_1\cdots a_j}  \\+ \sum_{d'+1 \le a_1 < \dots < a_{j-1} \le 2d'} \frac{1}{a_1\cdots a_{j-1}}\frac{1}{2d'+1}  
\end{align*}

\begin{align*}
   \hspace{-2em}\sum_{1\le b_1 < \dots < b_j \le d'} \frac{1}{b_1\cdots b_j} = \sum_{1\le b_1 < \dots < b_j \le d'-1} \frac{1}{b_1\cdots b_j} \\+ \sum_{1\le b_1 < \dots < b_{j-1} \le d'-1} \frac{1}{b_1\cdots b_{j-1}}\frac{1}{d'}, 
\end{align*}
the induction assumption on $d'$ gives that 
\begin{align*}
 \hspace{-2em}\sum_{1\le b_1 < \dots < b_j \le d'-1} \frac{1}{b_1\cdots b_j} \ge \sum_{d' \le a_1 < \dots < a_j \le 2d'-1} \frac{1}{a_1\cdots a_j} \\ \ge  \sum_{d'+1 \le a_1 < \dots < a_j \le 2d'} \frac{1}{a_1\cdots a_j},   
\end{align*}
and 
\begin{align*}
\hspace{-2em}\sum_{1\le b_1 < \dots < b_{j-1} \le d'-1} \frac{1}{b_1\cdots b_{j-1}} \frac{1}{d'} \ge \sum_{d' \le a_1 < \dots < a_{j-1} \le 2d'-1} \frac{1}{a_1\cdots a_{j-1}}\frac{1}{d'} \\\ge \hspace{-1em}\sum_{d'+1 \le a_1 < \dots < a_{j-1} \le 2d'} \frac{1}{a_1\cdots a_{j-1}}\frac{1}{2d'+1},    
\end{align*}
which gives the desired inequality.
\end{proof}

From Lemma \ref{lemma: degree-increase1} one sees that if the inequality $g_{d,d'}(n-1)\ge g_{d,d'}(x)$ holds for any $d$, then it is also true for higher values of $d$. Thus if we are able to prove the inequality for $d'=d-1$, this solves it for all possible $d$. This observation, together with Proposition \ref{prop: negative-coeff-d-1}, which deals with the case $d'=d-1$, is what enables us to prove our asymptotic result, Theorem \ref{thm:asymto}.

\begin{proof}[Proof of Theorem \ref{thm:asymto}]
Fix any integer $d'$ and consider first $d=d'+1$.
Since $g_{d,d'}(x)$ is of degree $d$ with positive coefficient of $x^d$, see Proposition \ref{prop: negative-coeff-d-1}, there exists $N=N(d') \in \mathbb{N}$ such that
$g_{d,d'}(y)\ge g_{d,d'}(z) \ge g_{d,d'}(N-1) \ge 0$
for all $y\ge z\ge N-1$. By Proposition \ref{prop: negative-coeff-d-1}, the polynomial $g_{d,d'}$ has exactly one sign-change in its coefficients, and combined with Lemma \ref{lemma: sign-change} we get that $g_{d,d'}(n-1) \ge g_{d,d'}(x)$ for all $x \in [0,n-1]$, for every $n \ge N$. It follows by Lemma \ref{lemma: degree-increase1} that the same is true for any degree $d > d'$. Hence Theorem \ref{thm: f-inequality} applies for any $d> d'$, when $n\ge N$, and it follows that, for any $r$, there is an ideal $I$ in $R$ generated by $r$ forms of degree $d$ such that $\dim_k(R/I)_{d+d'}$ equals the value conjectured by Fr\"oberg. 

Note that for a fixed $n,d$ and $r$, this property is generic for each $d'$. So if we fix $d$, then, as seen above, for each $d'<d$ there exists a $N(d')$ such that for any $r$, when $n\ge N(d')$, a generic ideal generated by $r$ forms of degree $d$ has the Hilbert function conjectured by Fr\"oberg in degree $d+d'$. Taking $N'$ to be the maximum of these $N(d')$ we get that, for any $r$, when $n\ge N'$, a generic ideal generated by $r$ forms of degree $d$ has the conjectured values of Hilbert function up to degree $2d-1.$
\end{proof}

Recall that the coefficients of $g_{d,d'}$ are increasing with respect to $d$, see Remark \ref{remark: increase d, decrease d'}. Thus, for each $d'$, it suffices to check the number of sign changes of $g_{d,d'}$ in a finite number of cases to conclude that there occurs at most one sign-change. Using this, we obtain an explicit condition, for low values of $d'$, which is sufficient to prove the Fr\"oberg conjecture for ideals generated by forms of degree $d$, up to degree $d+d'.$

\begin{theorem}\label{thm: low d'}
Fix $d>d'>0$ and $n$ such that $\dim_{k}R_{d+d'} \ge (\dim_k R_{d'})^2$.  If $d' \le \compBound$, then for any $r$ there is an ideal $I$ generated by $r$ forms of degree $d$ that has the Hilbert series conjectured by Fr\"oberg, up to degree $d+d'$.
\end{theorem}

\begin{proof} 
We begin by checking that $g_{d,d'}(x)$ has at most one sign-change in its coefficients when $d'\le \compBound$. From Remark \ref{remark: increase d, decrease d'}, we see that each coefficient of $g_{d,d'}(x)$ is increasing with respect to $d$ and for suffiently large $d$, all coefficients are positive. Hence it is sufficient to check the number sing-changes for a finite number of values $d$ for each value of $d'$. These checks have been done using MATLAB~\cite{MATLAB}.

Let now $d,d'$ and $n$ be fixed integers such that $d>d'$, $d'\le \compBound$ and $\dim_k R_{d+d'} \ge (\dim_k R_{d'})^2$.
 We have that 
$\dim_k R_{d+d'} \ge (\dim_k R_{d'})^2$ is equivalent to $g_{d,d'}(n-1)\ge 0$ by Lemma \ref{lemma:equiv-}, and from Remark \ref{remark: increase d, decrease d'} we see that $g_{d,d'}(n-1)$ is decreasing with respect to $d'$. Thus  $\dim_k R_{d+d''}\ge (\dim_k R_{d''})^2$,  for all $d''\le d'$. Since $g_{d,d''}$ has at most one sign-change in its coefficients, it follows now by Theorem \ref{thm: sign-changes} that for each $d''\le d'$, and any $r$, there is an ideal in $R$ generated by $r$ forms in degree $d$ that has the Hilbert function conjectured by Fr\"oberg in degree $d+d''$. For each $d''\le d'$, having the Hilbert function conjectured by Fr\"oberg in degree $d+d''$ is a generic property for ideals generated by $r$ forms of degree $d$ in $R$. Thus, for all $r$, there exists an ideal $I$ in $R$ generated by  $r$ forms of degree $d$ such that $\dim_k (R/I)_i$ equals the value conjectured by Fr\"oberg for  $i\leq d+d'$.
\end{proof} 

Observe that the bound $d'\le \compBound$ depends only on a numerical calculation, it can easily be extended further with more calculations.  

\section{Discussions} \label{sec:discussions}
Since Aubry \cite{Aubry1995392}, Migliore and Mir\'o-Roig \cite{MIGLIORE200379}, and Nenashev \cite{NENASHEV2017272} have different criteria for when the series is correct in degree $d+d'$, for $d'>1$, we find it relevant to explain the relation between these and our results. Firstly, Theorem~\ref{thm:d2} completely generalize the result by Hochster and Laksov from multiplication by linear forms to multiplication by quadratic forms. 
 Secondly, Theorem \ref{thm:asymto} shows that for ideals generated in degree $d$, the Fr\"oberg conjecture is true up to degree $2d-1$, as long as the number of variables $n$ is sufficiently large. In particular, this settles the conjecture up to degree $2d-1$ in all but finitely many cases, which is not true for any of the results of Aubry \cite{Aubry1995392}, Migliore and Mir\'o-Roig \cite{MIGLIORE200379}, or Nenashev \cite{NENASHEV2017272}. Below, we give a short description for each of these results, and comment briefly on the relation to our work.

Similar to us, Aubry \cite{Aubry1995392} adapted the technique by Hochster and Laksov to handle the case $d+d'$ for some $d'<d$, depending on $d'$ and $n$. This result is given using a function $\delta(n-1,d')$, and states that whenever $d\ge \delta(n-1,d')$, then, for any $r$, there exists an ideal $(f_1,\dots,f_r)$ generated in degree $d$ that has the Hilbert function conjectured by Fr\"oberg in degree $d+d'$. Lemma 5.8 in \cite{Aubry1995392} shows that 
$$
d\ge \delta(n-1,d') \implies \dim_k R_{d+d'} \ge (\dim_k R_{d'})^2.
$$ Thus, for $d'\le \compBound$, our Theorem \ref{thm: low d'} is a direct generalization of the main result of \cite{Aubry1995392}. Moreover for a fixed $d'$ the function $\delta(n-1,d')$ tends to infinity when $n$ does, so in particular it can not be used to obtain any asymptotic result similar to our Theorem \ref{thm:asymto}. Thus, for fixed $d'$ and $d>d$, it leaves infinitely many cases open.

The result by Migliore and Mir\'o-Roig is built on induction using Anick's result for three variables. When $d'<d$, they conclude that ideals generated by $r$ general forms of degree $d$ have the expected Hilbert function in degree $d+d'$ for
\[
r \le \frac{\left(d'+d+1\right)\left(d'+d+2\right)}{\left(d'+1\right)\left(d'+2\right)}.
\]
When $d'=2$, this gives $r\le (d+4)(d+3)/12$. Note that this bound is independent of $n$, but it covers smaller proportion of the relevant cases when $n$ grows.

In \cite{NENASHEV2017272}, Nenashev shows for fixed $d'$ and $d>d'$, that when $r\le \frac{\dim_k R_{d+d'}}{\dim_k R_{d'}} - \dim_k R_{d'}$ or $r\ge \frac{\dim_k R_{d+d'}}{\dim_k R_{d'}} + \dim_k R_{d'}$ an ideal generated by $r$ general forms of degree $d$ have the expected Hilbert function in degree $d+d'$. Thus, covering all but $2\dim_k R_{d'}$ cases. When the number of variables $n$ increases, so does $2\dim_k R_{d'}$, and so also this result leaves infinitely many cases open for fixed $d'$ and $d>d'$. It is important to note that the critical value for $r$, when we expect the forms to span exactly $R_{d+d'}$, is in the center of the interval that is not covered by Nenashev's result. Note moreover that $$\frac{\dim_k R_{d+d'}}{\dim_k R_{d'}} - \dim_k R_{d'} \ge 0 \iff \dim_k R_{d+d'} \ge (\dim_k R_{d'})^2.$$ Thus, when $d'\le \compBound$, the result that the conjecture is true when  $r\le \frac{\dim_k R_{d+d'}}{\dim_k R_{d'}} - \dim_k R_{d'}$ is directly generalized by Theorem \ref{thm: low d'}.

Our results are all given under the condition that $\dim_k R_{d+d'} \ge (\dim_k R_{d'})^2$, so it is natural to ask if this is a limitation of the technique, or a consequence of sub-optimal approximations. We claim, that without any specific information on the irreducible components of $V$, this is (almost) the best one can obtain.
 Indeed, if $\dim_{k}R_{d+d'} < (\dim_k R_{d'})^2$, it follows that  $r< \dim_k R_{d'}$, and hence for any $r$-tuple of forms $(a_1,\dots,a_r)$ in $R_{d'}^{r}$ we have that $\HF_{R/(a_1,\dots,a_r)}(d') >0$ and with no further assumption on $(a_1,\dots,a_r)$ we cannot conclude that $\HF_{R/(a_1,\dots,a_r)}(d+d') =0$. Even for $P_t$, where $t=\dim_k R_d' - r$ and the forms are maximally independent, the best upper bound on $\dim_k V_t$ we can obtain is equal to $r\dim_k R_{d}  +  \HF_{R/(a_1,\dots,a_r)}(d+d') -1$, which is not enough to show that $\dim V_t < r\dim_k R_d$. This means that it is not a collection of approximations that has lead us to the restriction $\dim_{k}R_{d+d'} \ge (\dim_k R_{d'})^2$. Instead,  without any more information about the irreducible components of $V$, it is a necessary condition to prove that $\dim_k V < \dim_k R_d$.

It is important to note that the method used here cannot be used in the case when $d'\ge d$ since in this case, there are Koszul relations, so we cannot prove that the forms spanned by the generators in degree $d+d'$ are linearly independent. To move further in this direction, we believe that a completely different approach is needed.

\subsection*{Acknowledgement}
This work was initiated in connection with a reading seminar on the Fr\"oberg conjecture in Stockholm during the academic year 2024/2025. The authors would like to thank the participants of this seminar, and we are especially grateful to Clas L\"ofwall who presented the proof of the Hochster-Laksov result at the seminar.

The first author was supported by the Swedish Research Council grant VR2024-04853. The second and the third author were supported by the Swedish Research Council grant VR2022-04009.

\bibliographystyle{abbrv}
\bibliography{references.bib}

\end{document}